\documentclass[12pt]{article}
\usepackage{geometry}
\usepackage[utf8]{inputenc}
\usepackage{amsthm,amsfonts}
\usepackage[sumlimits]{amsmath}
\usepackage{amssymb}
\usepackage[all]{xy}
\usepackage{tikz}
\usepackage{graphicx}
\usepackage{url}
\usepackage{tikz-cd}
\usepackage{pgf}
\usepackage{ragged2e}
\usepackage{hyperref}
\usepackage[T1]{fontenc}
\usepackage{lmodern}
\usepackage{mathtools}

\newtheorem{lema}{Lemma}[section]
\newtheorem{teo}[lema]{Theorem}
\newtheorem{cor}[lema]{Corollary}
\newtheorem{prop}[lema]{Proposition}
\theoremstyle{definition}

\newtheorem{ex}[lema]{Example}

\newtheorem{que}[lema]{Question}
\numberwithin{equation}{section}

\newcommand{\cf}{\emph{cf.} }
\newcommand{\tq}{\,|\,}
\newcommand{\ie}{\emph{i.e.}}

\newcommand{\apres}[2]{\langle #1 \tq #2\rangle}

\newcommand{\ed}{\ar@{-}}

\renewcommand{\epsilon}{\varepsilon}

\newcommand{\lgoth}{\mathcal{L}}

\numberwithin{equation}{section}
\newcommand{\BR}{\mathbb{R}}

\newcommand{\BN}{\mathbb{N}}
\newcommand{\BZ}{\mathbb{Z}}

\DeclareMathOperator{\Aut}{Aut}
\DeclareMathOperator{\Id}{Id}

\DeclareMathOperator{\im}{Im}
\DeclareMathOperator{\rk}{rk}

\DeclareMathOperator{\Fix}{Fix}\DeclareMathOperator{\FGFPa}{FGFP_a}\DeclareMathOperator{\FnFPa}{F_nFP_a}\DeclareMathOperator{\FPnFPa}{FP_nFP_a}

\renewcommand{\F}{\operatorname{F}}
\renewcommand{\FP}{\operatorname{FP}}

\newcommand{\Cay}{\operatorname{Cay}}

\theoremstyle{plain}
\newtheorem{mainthm}{Theorem}
\newtheorem{maincor}{Corollary}

\begin{document}

\title{On the finiteness properties of fixed subgroups of automorphisms}

\author{Kisnney Emiliano de Almeida\\
Universidade Estadual de Feira de Santana,\\
Departamento de Ciências Exatas,\\
Av. Transnordestina S/N, CEP 44036-900 - Feira de Santana - BA - Brazil.\\
e-mail:~\texttt{kisnney@gmail.com}\vspace*{4mm}\\
Luis Mendonça ~\\
Universidade Federal de Minas Gerais,\\
Departamento de Matem\'atica,\\
Av. Antônio Carlos, 6627, CEP 31270-901 - Belo Horizonte - MG - Brazil.\\
e-mail:~\texttt{luismendonca@mat.ufmg.br}
}

\date{\today}

\maketitle

\begin{abstract}
    We use Sigma-invariants to study homotopical and homological finiteness properties of fixed subgroups of automorphisms of a group $G$ in terms of  its center $Z(G)$ and the induced automorphisms on its associated quotient $G/Z(G)$. Specializing to the case where the center is a direct factor of the group, we answer a question made by Lei, Ma and Zhang.
\end{abstract}

\section{Introduction}
Given a group $G$ and an automorphism $\phi \in \text{Aut}(G)$, the subgroup of fixed points, $$\Fix \phi = \{g \in G \mid \phi(g) = g\},$$ is an object of fundamental study. It encodes the symmetry of $G$ under the action of $\phi$ and its internal structure reveals deep information about the group $G$ itself.  

There has been a wide interest in fixed subgroups of finitely generated free groups: Gersten \cite{Gersten} has proven them to be always finitely generated and Bestvina and Handel proved that the rank of $\Fix\phi$ is uniformly bounded by the rank of the ambient free group \cite{bestvina1992train}, confirming a conjecture of Scott from the 70s. 

With that in mind, Lei, Ma and Zhang \cite{LeiMaZhang} have defined that a group $G$ has $\FGFPa$ property if $\Fix \phi$ is finitely generated for all $\phi \in \Aut(G)$.  Besides from free groups, Minasyan and Osin \cite{minasyan2012fixed} have showed that this property holds for limit groups  and Zhang, Ventura and Wu \cite{ZhangVenturaWu2015} have proved that it holds for finite direct products of non-abelian free groups, among other classes. 

It is not always true that fixed subgroups of finitely generated groups are finitely generated, even for direct products of $\FGFPa$ groups. A simple example is given by the automorphism $\phi \in \Aut(F_2 \times \BZ)$ given by $\phi(g,n) = (g, \alpha(g) + n)$, where $\alpha \colon F_2 \to \BZ$ sends all elements in a free basis of $F_2$ to $1$. In this case, $\Fix \phi = \ker \alpha \times \BZ$, which is not finitely generated.

Our main goal in this paper is to study finiteness properties $\F_n$ and $\FP_n$ of $\Fix\phi$, for $\phi$ being an automorphism of a given group $G$, in terms of its center $Z(G)$ and the quotient $G/Z(G)$. These finiteness properties generalize the concepts of finitely generated groups - indeed a group $G$ is of type $\F_1$ if and only if $G$ is of type $\FP_1$ if and only if $G$ is finitely generated; also $G$ is finitely presented if and only if $G$ is of type $\F_2$, which implies type $\FP_2$. We will explain more about these finiteness properties in Section \ref{sec.prel}.

For a finitely generated group $G$, its BNS-invariant $\Sigma^1(G)$ is a certain subset of the character sphere $S(G)$; the latter is formed  by the classes $[\chi]$ of non-trivial homomorphisms $\chi\colon G \to \BR$, under the equivalence relation where $\chi_1 \sim r \chi_1$ if $r\in \BR_{>0}$. Its main application is to determine which subgroups of $G$ above the commutator $G'$ are finitely generated \cite{BNS1987}. There are also higher topological and homological versions $\Sigma^n(G)$ and $\Sigma^n(G,\BZ)$ which may be similarly used to determine if those subgroups inherit the $\F_n$ and $\FP_n$ properties from the group $G$ \cite{BieriRenz1988}- we give more details about them in Section \ref{sec.prel}. 

Generalizing the $F_2\times \BZ$ example above, Lei, Ma and Zhang \cite{LeiMaZhang} considered direct products of the form $G \times A$, where $A$ is free abelian of finite rank. If $Z(G)$ is trivial, then all automorphism of such a group are of the form \[ \phi(g,a) = (\psi(g), \alpha(g) + \gamma(a)),\] where $\psi \in \Aut(G)$, $\gamma \in \Aut(A)$ and $\alpha \colon G \to A$ is a homomorphism. The homomorphism $\alpha$ turns out to have strong influence in the finiteness properties of the fixed subgroup $\Fix \phi$, and this information is captured by studying the BNS-invariant of the group $G$.

A group $H$ is said to be  weakly Howson if the intersection of two finitely generated subgroups $A, B \leq H$, one of them being normal in $H$, is always finitely generated. Lei, Ma and Zhang proved the following.

\begin{teo}[\cite{LeiMaZhang}] \label{teoLeiMaZhang}
Let $H$ be a weakly Howson group with trivial center.
\begin{enumerate}
    \item $H \times \BZ$ has $\FGFPa$ if and only if $H$ has $\FGFPa$ and $\Sigma^1(H)$ contains all classes $[\chi]$ of homomorphisms with $\rk_{\BZ} \im \chi = 1$.
    \item  $H \times \BZ^m$ has $\FGFPa$ for all $m \geq 1$ if and only if $H$ has $\FGFPa$ and $H'$ is finitely generated.
\end{enumerate}
\end{teo}

Inspired by the result above, the authors formulated the following question.

\begin{que}\cite{LeiMaZhang}\label{question}
    Does $H\times \BZ$ have $\FGFPa$ if the group $H$ has $\FGFPa$ and $H'$ is finitely generated?
\end{que}

Our main result is the following.

\begin{mainthm}\label{Fix.Center}
 Let $n\in\BN$ and $G$ be a group of type $\F_n$ with finitely generated center. Let $\phi \in \Aut(G)$, $\bar{\phi}$ the automorphism of $G/Z(G)$ induced by $\phi$ and \[ I_\phi = \{ z^{-1} \phi(z) \mid z \in Z(G)\} \leq Z(G).\] Then the following statements are equivalent:
 \begin{enumerate}
  \item[(i)]$\Fix \phi$ is of type $\F_n$ (resp. $\FP_n$),
  \item[(ii)] Both $\Fix \bar{\phi}$ and its subgroup $P_\phi = \{ gZ(G) \in G/Z(G) \mid g^{-1} \phi(g) \in I_\phi\} \lhd \Fix\bar{\phi}$ are of type $\F_n$ (resp. $\FP_n$),
  \item[(iii)] $\Fix \bar{\phi}$ is of type $\F_n$ (resp. $\FP_n$) and for all $[\chi] \in \Sigma^1(\Fix \bar{\phi} )^c$ (resp. $\Sigma^1(\Fix \bar{\phi}, \BZ)^c$)  there exists $g \in G$ such that $g^{-1} \phi(g) \in I_\phi$ and $\chi(gZ(G)) \neq 0$.
 \end{enumerate}
\end{mainthm}

An interesting case is to consider only automorphisms of finite order. For example, Kochloukova, Martínez-Pérez, Nucinkis \cite{KMPN2012} have shown that the fixed points of the finite order automorphisms of the generalized Thompson's groups are finitely generated if and only if they are of type $\F_n$ for all $n$; they also prove the latter is actually true for the Thompson's group $\F$.

Roy and Ventura \cite{MallikaEnric} proved that fixed subgroups of finite order 
automorphisms of $F_n \times \BZ^m$ are always finitely generated - although that is not true for all automorphisms, as we have mentioned. An application of Theorem \ref{Fix.Center} gives the following generalization.

\begin{maincor}\label{finiteorder}
Let $G$ be a group of type $\F_n$ with finitely generated center, let $\phi \in \Aut(G)$ be an automorphism of finite order and let $\bar{\phi}$ the automorphism of $G/Z(G)$ induced by $\phi$. Then $\Fix \phi$ is of type $\F_n$ (resp. $\FP_n$)  if and and only if $\Fix \bar{\phi}$ is of type $\F_n$ (resp. $\FP_n$). 
\end{maincor}

We say that a group has property $\FnFPa$ (resp. $\FPnFPa$) if $\Fix\phi$ is of type $\F_n$ (resp. $\FP_n$) for all $\phi\in \Aut G$. Note that $\FGFPa = \FnFPa = \FPnFPa$ for $n=1$. The theorem below is a criterion which analyzes the properties above from the correlate finiteness properties of some kernels. We use the notation $I_{\phi}$ as in Theorem \ref{Fix.Center}.

\begin{mainthm}\label{FnFPaCenter}
 Let $G$ be a group of type $\F_n$  with finitely generated center. Then $G$ satisfies $\FnFPa$ (resp. $\FPnFPa$) if, and only if, for every homomorphism $\nu \colon G/Z(G) \to Z(G)$ and for all $\phi \in \Aut(G)$, the kernel of the map
 \[\theta \colon \Fix \bar{\phi} \to Z(G)/I_\phi, \  \ gZ(G) \mapsto g^{-1}\phi(g) \nu(gZ(G)) I_\phi\]  is of type $\F_n$ (resp. $\FP_n$).
\end{mainthm}

The following corollary gives us a glance of what these properties demand of subgroups above the commutator.

\begin{maincor}\label{corfnfpa.sigma}
    Let $G$ be a group with the $\FnFPa$ (resp. $\FPnFPa$) property and finitely generated center. If $G' \leq N \leq G$ satisfies $\rk_{\BZ} G/N \leq \rk_{\BZ} Z(G)$, then $N$ is of type $\F_n$ (resp. $\FP_n$). 
\end{maincor}

Aiming to answer Question \ref{question}, we use Theorem \ref{FnFPaCenter} to establish the following result for when the center of the group is a direct factor.

\begin{mainthm}\label{cor.dpdim}
Let $H$ be a centerless group and let $A$ be a finitely generated abelian group. Then the following are equivalent:
\begin{enumerate}
    \item  $G\coloneqq H \times A$ has $\FnFPa$ (resp. $\FPnFPa$) property;
    \item  $H$  has $\FnFPa$ (resp. $\FPnFPa$) property and $\ker(\chi _{|\Fix \psi})$ is of type $\F_n$ (resp. $\FP_n$) for every homomorphism $\chi \colon H \to \BR$
    such that $\rk_{\BZ} \im \chi\leq \rk_{\BZ} A$ and for all $\psi \in \Aut(H)$.
\end{enumerate}
\end{mainthm}

By using Theorem C we are able to find two examples of groups that give a negative answer to Question \ref{question}.

This paper is structured as follows. In Section \ref{sec.prel} we will establish some preliminary results we need. In Section \ref{sec.gen} we prove Theorem \ref{Fix.Center} and Corollary \ref{finiteorder}; in Section \ref{sec.fgfpa} we prove Theorem \ref{FnFPaCenter} and Corollary \ref{corfnfpa.sigma}; in Section \ref{sec.dp} we study the case where the center is a direct factor and prove Theorem \ref{cor.dpdim}. Finally, in Section \ref{sec.ex} we answer Question \ref{question}.

\section{Preliminaries}\label{sec.prel}
A group $G$ is said to be of type $\F_n$ if there is a $K(G,1)$-complex with finite $n$-skeleton. It is well known that $\F_1$ is equivalent with $G$ being finitely generated, and $\F_2$ coincides with $G$ being finitely presentable. 

For an arbitrary ring $R$, an $R$-module $A$ is said to be of type $\FP_n$ if it admits a projective resolution
\[ \cdots \to P_2 \to P_1 \to P_0 \to A \to 0\]
with $P_k$ finitely generated for all $k \leq n$. Specializing to $R = \BZ G$ and $A = \BZ$, we obtain the definition of a group of type $\FP_n$.

Again, $\FP_1$ coincides with $G$ being finitely generated, but $\FP_2$ is strictly weaker than finitely presentability, as shown by Bestvina and Brady \cite{BestvinaBrady1997}. It is also well known that $\F_n$ implies $\FP_n$ and that $\F_{n+1}$ (resp. $\FP_{n+1}$) implies $\F_n$ (resp. $\FP_n$) for all $n$. We also say a group is of type $\F_{\infty}$ (resp. $\FP_{\infty}$) if it  is of type $\F_n$ (resp. $\FP_n$) for all $n$. The easiest examples of groups of type $\F_{\infty}$ are finitely generated free groups and finitely generated abelian groups.

We recall some other well known results about these properties.
\begin{prop}\label{basicfinprop}
    Let $1\to A\to B\to C\to 1$ be a short exact sequence of groups. 
    \begin{enumerate}
        \item If $A$ and $C$ are of type $\F_n$ (resp. $\FP_n$) then $B$ is of type $\F_n$ (resp. $\FP_n$);
        \item 
        If $A$ is of type $\F_{n-1}$ (resp. $\FP_{n-1}$) and $B$ is of type $\F_n$ (resp. $\FP_n$) then $C$ is of type $\F_n$ (resp. $\FP_n$).
           \end{enumerate}
\end{prop}
\begin{proof}
    \cf \cite{geoghegan2008topological}
\end{proof}
It is also known that properties $\F_n$ and $\FP_n$ pass to and from finite index subgroups. For more general information about $\F_n$ and $\FP_n$ properties we refer the reader to \cite{Bieri, Brown, geoghegan2008topological}.

Next, we define the $\Sigma$-invariants. For $G$ being a finitely generated group, its character sphere $S(G)$ is the set of non-zero homomorphisms  $\chi \colon G \to \BR$ modulo the equivalence relation where $\chi_1 \sim \chi_2$ when $\chi_2 = r \chi_1$ for some $r \in \BR_{>0}$.  For $\chi \colon G \to \BR$ we define the submonoid $G_{\chi} = \{ g \in G \mid \chi(g) \geq 0 \}$, and the homological $\Sigma$-invariants are defined simply as
\[\Sigma^n(G,\BZ) = \{[\chi] \in S(G) \mid \BZ \textrm{ is of type $\FP_n$ as $\BZ G_{\chi}$-module}\}.\]

For the homotopical counterparts $\Sigma^n(G)$, we will define just $\Sigma^1$ and $\Sigma^2$ and resort to the formula $\Sigma^n(G) = \Sigma^2(G) \cap \Sigma^n(G,\BZ)$ for $n \geq 2$ (\cite{BieriRenz1988}). 

Let $X$ be a finite generating set of $G$, and let $\Cay(G,X)$ be the associated Cayley graph. For $[\chi] \in S(G)$, we consider the full subgraph $\Cay(G,X)_{\chi}$ spanned by the vertices in $G_\chi$. We put
\[\Sigma^1(G) = \{ [\chi] \in S(G) \mid \Cay(G,X)_{\chi} \hbox{ is a connected graph}\}.
\]
We can define similarly the invariant $\Sigma^2(G)$. Suppose that $G$ is finitely presented and let $\mathcal{C}$ be the Cayley complex of $G$ associated with a finite presentation $G = \langle X \mid R \rangle$. For any character $\chi \colon G \to \BR$, the subset $G_{\chi} \subset G$ determines a full subcomplex $\mathcal{C}_{\chi}$ of $\mathcal{C}$. By definition
\[
\Sigma^2(G) = \{ [\chi] \in S(G) \mid \mathcal{C}_{\chi} \hbox{ is 1-connected for some finite presentation } \langle X \mid R\rangle  \hbox{ of } G\}.
\]

We say that a character $[\chi]\in S(G)$ is discrete if $\im\chi\simeq \BZ$. In the following theorem, we collect some basic results on the $\Sigma$-invariants that we need.

\begin{teo}[\cite{BNS1987, BieriRenz1988}]\label{teo.sigma}\label{sigma1center}\label{sigma1kerchar}\label{sigma1app}\label{teo.sigma.free}
    Let $G$ be a group of type $\F_n$  and let $\chi \colon G\to \BR$ be a non-trivial homomorphism.

\begin{enumerate}

   \item 
    Let $H$ be a subgroup of $G$ containing $G'$. Then $H$ is of type $\F_n$ if and only if
    \[S(G,H) \coloneqq \{[\chi]\in S(G) \tq \chi(H)=0\}\subset \Sigma^n(G).\]
     In particular,  $S(G)=\Sigma^n(G)$ if and only if $G'$ is of type $\F_n$;

    \item 
    Suppose $[\chi]$ is discrete. Then  $\ker \chi$ is of type $\F_n$   if and only if 
    $\{\chi, -\chi\}\subset \Sigma^n(G)$; 

    \item 
    If $H\leq G$ is a subgroup of finite index then $[\chi_{|H}] \in \Sigma^n(H)$ if and only if  $[\chi] \in \Sigma^n(G)$;

   \item 
    If $\chi(Z(G))\neq 0$ then $[\chi]\in \Sigma^n(G)$;

     \item 
     If $G$ is free then $\Sigma^n(G)=\emptyset$.
  
\end{enumerate}
    
\end{teo}

In Theorem \ref{teo.sigma}, we may replace $\F_n$ with $\FP_n$ and $\Sigma^n(G)$ with $\Sigma^n(G,\BZ)$ and find the appropriate homological counterparts.

\section{Fixed subgroups and the center}\label{sec.gen}

Let $G$ be a group. From now on, for $\phi \in \Aut(G)$ we will denote by $\bar{\phi}$ the automorphism of $G/Z(G)$ induced by $\phi$. Let
\[ I_\phi \coloneqq \{ z^{-1} \phi(z) \mid z \in Z(G)\} \subseteq Z(G).\]
Notice that $I_\phi$ is actually a subgroup of $Z(G)$, since if $z_1,z_2\in Z(G)$ then
$$z_1^{-1}\phi(z_1)\left(z_2^{-1}\phi(z_2)\right)^{-1}=(z_1z_2^{-1})^{-1}\phi(z_1z_2^{-1}).$$

We also define the map 
\begin{align*}
  \epsilon_{\phi} \colon \Fix\bar{\phi} &\to Z(G)/I_\phi\\
  gZ(g) &\mapsto g^{-1}\phi(g)I_\phi.
\end{align*}

Note that $\epsilon_\phi$ is well defined on $\Fix\bar{\phi}$, but not on $G/Z(G)$ in general. Indeed, for $gZ(G) \in \Fix \bar{\phi}$ we have $g^{-1}\phi(g) \in Z(G)$ and for $z \in Z(G)$ the elements $g^{-1} \phi(g)$ and $(gz)^{-1}\phi(gz)$ represent the same class modulo $I_\phi$. Moreover, using that the elements $\{g^{-1} \phi(g)\}$  are central in $G$, we have
 \[\epsilon_\phi( ghZ(G) ) = (gh)^{-1} \phi(gh)I_\phi = h^{-1} (g^{-1} \phi(g))\phi(h)I_\phi = (g^{-1} \phi(g)) (h^{-1} \phi(h)) I_\phi,\]
  for $gZ(G),hZ(G) \in \Fix\bar{\phi}$, so $\epsilon_\phi$ is a homomorphism.

Note that $P_\phi \coloneqq  \{ gZ(G) \in G/Z(G) \mid g^{-1} \phi(g) \in I_\phi\}=\ker \epsilon_{\phi} \lhd \Fix\bar{\phi}$.

\begin{proof}[Proof of Theorem \ref{Fix.Center}]

We prove the topological version since the homological one is similar.

 Denote by $\pi \colon G \to G/Z(G)$ the canonical projection. We have an exact sequence
 \[ 1 \to Z(G) \cap \Fix\phi \to \Fix\phi \to \pi(\Fix \phi) \to 1.\]
 Since $Z(G) \cap \Fix \phi\leq Z(G)$ is finitely generated abelian, it follows from Proposition \ref{basicfinprop} that $\Fix \phi$ is $\F_n$ if and only if $\pi(\Fix \phi)$ is so.

 We have by construction
 \[ \pi(\Fix \phi)=\{ gZ(G) \in G/Z(G) \mid \exists z \in Z(G) \text{ such that } \phi(gz) = gz\}.\]
 In the situation above, $g^{-1} \phi(g) = z \phi(z)^{-1} = (z^{-1})^{-1} \phi(z^{-1}) \in I_\phi$, so $\pi(\Fix \phi)=P_\phi$.

As $\im \epsilon_\phi$ is finitely generated abelian, $P_\phi=\ker \epsilon_{\phi}$ being $\F_n$ implies that $\Fix \bar{\phi}$ is too, by Proposition \ref{basicfinprop}. So $(i)$ and $(ii)$ are equivalent.

 The equivalence of $(ii)$ and $(iii)$ follows from Theorem~\ref{sigma1app}: the subgroup $P_\phi$ is $\F_n$ if and only if for all $[\chi] \in \Sigma^1(\Fix\bar{\phi})^c$ there is $p \in P_\phi$ such that $\chi(p) \neq 0$, that is, there is $g \in G$ such that $\chi(gZ(G)) \neq 0$ and $g^{-1}\phi(g) \in I_\phi$.
 \end{proof}

\begin{proof}[Proof of Corollary \ref{finiteorder}]
By Theorem~\ref{Fix.Center} it is enough to show that if $\phi$ is of finite order and $\Fix \bar{\phi}$ is of type $\F_n$ then  $P_\phi$ is a finite index subgroup of $\Fix \bar{\phi}$.

First notice that $z^{-1} \phi^k(z) \in I_\phi$ for all $k \geq 1$ and $z \in Z(G)$. For $k=1$ this is just the definition, and for $k>1$ we use induction: $z^{-1} \phi^k(z) = z^{-1} \phi^{k-1}(z) z_2^{-1} \phi(z_2) \in I_\phi$, where $z_2 = \phi^{k-1}(z) \in Z(G)$.

Now assume that $\phi^m = \Id$. If $gZ(G) \in \Fix \bar{\phi}$ (so that $g^{-1} \phi(g) \in Z(G)$), we have:
\begin{align*}
   1 = g^{-1} \phi^m(g) &= g^{-1} \phi(g) \phi(g^{-1}) \phi^2(g) \phi^2(g^{-1}) \cdots \phi^{m-1}(g) \phi^{m-1}(g^{-1}) \phi^m(g) \\
   & = z \phi(z) \phi^2(z) \cdots \phi^{m-1}(z),
\end{align*}
where $z = g^{-1}\phi(g) \in Z(G)$. It follows then that
\[z^{-m} = z^{-1} \phi(z) \cdot z^{-1}\phi^2(z) \cdots z^{-1} \phi^{m-1}(z) \in I_\phi.\]
Thus for all $gZ(G) \in \Fix \bar{\phi}$ we have \[\epsilon_\phi(g^m Z(G)) = \epsilon_\phi(gZ(G))^m = (g^{-1}\phi(g))^m I_\phi = I_\phi.\] 
This proves that $\im \epsilon_\phi$ is an abelian group of exponent at most $m$. It is also finitely generated, as it is a quotient of $\Fix \bar{\phi}$, thus it is finite. So $P_\phi = \ker(\epsilon_\phi)$ has finite index in $\Fix \bar{\phi}$.
\end{proof}
%As $P_\phi = \ker \epsilon_\phi$, this proves that  $[\Fix \bar{\psi}: P_\phi]$ is at most $n$.
\section{Property \texorpdfstring{$\FGFPa$}{FGFPa} and generalizations}\label{sec.fgfpa}

 \begin{proof}[Proof of Theorem \ref{FnFPaCenter}]
 Again we prove only the topological version. Suppose the statement about kernels is true and let $\phi\in\Aut G$. Note that $\theta = \epsilon_\phi + \pi\circ \nu_{|\Fix\bar{\phi}}$, where $\pi \colon Z(G)\to Z(G)/I_{\phi}$ is the projection.  By taking $\nu$ to be the trivial homomorphism, we have that $P_\phi = \ker(\epsilon_\phi)$ is of type $\F_n$. Since $P_{\phi}=\ker \epsilon_{\phi}$ and $\im\phi$ is finitely generated abelian then $\Fix \bar{\phi}$ is also of type $\F_n$ by Theorem \ref{basicfinprop}, hence Theorem~\ref{Fix.Center} implies $\Fix \phi$ is of type $\F_n$. Since that is true for all $\phi\in\Aut G$, then $G$ satisfies $\FnFPa$.

 Conversely, assume that $G$ has $\FnFPa$. Let $\phi\in \Aut G$ and $\nu \colon G/Z(G) \to Z(G)$ be a homomorphism. Denote by $\mu \colon G \to Z(G)$ its lift to $G$, and consider the homomorphism given by
 \[ \psi \colon G \to G, \  \ \psi(g)= \phi(g) \mu(g).\]
 It has an inverse given by the map
 \[ \eta \colon G \to G, \  \ \eta (g) \coloneqq \phi^{-1}(g)  \phi^{-1} \mu  \phi^{-1}(g^{-1}). \]
Indeed, using that $\mu(z) = 1$ for all $z \in Z(G)$, so that in particular $\mu \phi^{-1} \mu(g) = 1$ for all $g \in G$, we have
\begin{align*}
 \eta \psi (g) &= \eta(\phi(g)\mu(g)) \\
  & = \phi^{-1}(\phi(g)) (\phi^{-1} \mu  \phi^{-1}(\phi(g)))^{-1}   \cdot \phi^{-1}(\mu(g)) (\phi^{-1} \mu  \phi^{-1}(\mu(g)))^{-1} \\
 & = g (\phi^{-1} \mu  (g))^{-1} \cdot \phi^{-1}(\mu(g))\\
 & =g,
\end{align*}
and similarly $\psi \eta = \Id$. So $\psi \in \Aut(G)$. By hypothesis $\Fix(\psi)$ is of type $\F_n$, thus by Theorem~\ref{Fix.Center} so is $P_\psi$, where
\[ P_\psi = \{ gZ(G) \in G/Z(G) \mid g^{-1}\phi(g)\mu(g) \in I_\psi\} .\]
As $I_\phi = I_\psi$ and $\mu(g) = \nu(gZ(G))$, we see that $P_\psi= \ker(\theta)$.
\end{proof}

\begin{proof}[Proof of Corollary \ref{corfnfpa.sigma}]
We prove the topological version. Assuming $G$ has $\FnFPa$, let $\phi = \Id$ in Theorem~\ref{FnFPaCenter}. Then $I_\phi$ is the 
trivial subgroup, $\epsilon_\phi$ is the trivial map and $\Fix \bar{\phi} = G/Z(G)$, so the theorem's statement implies any 
homomorphism $\nu \colon G/Z(G) \to Z(G)$ has kernel of type $\F_n$.

Assuming $G'\leq N \leq G$ and  $\rk_{\BZ} G/N \leq \rk_{\BZ} Z(G)$, let $\chi \colon G \to \BR$ be a non-trivial homomorphism such that $\chi(N) = 0$. Then $\rk_{\BZ} \im \chi \leq \rk_{\BZ} G/N \leq \rk_{\BZ} Z(G)$.

If $\chi(Z(G)) \neq 0$, then $[\chi] \in \Sigma^n(G)$ by Theorem~\ref{sigma1center}. Otherwise, we consider the induced homomorphism $\bar{\chi} \colon G/Z(G) \to \BR$. By composing with an embedding $\iota  \colon \im \chi \to Z(G)$, we see that $\ker \bar{\chi}=\ker \iota \circ \bar{\chi}$ has type $\F_n$ by the beginning of the proof, and since $Z(G)$ is finitely generated, we find that $\ker \chi$ is of type $\F_n$ too, by Theorem \ref{basicfinprop}. Hence $[\chi] \in \Sigma^n(G)$ by Theorem \ref{teo.sigma}. 

Since $[\chi]$ was arbitrarily chosen, by Theorem~\ref{sigma1app} we find that $N$ is of type $\F_n$.
\end{proof}

\begin{ex}
    Consider $G$ as being the pure braid group (on two strings) of the Klein bottle, which may be written as $P_2(\mathbb{K})\simeq F_2\rtimes (\BZ \rtimes \BZ)$, the semidirect product of the free group $F_2=\langle x, y\rangle$ with $\BZ\rtimes \BZ=\langle a, b \tq ab=ba^{-1}\rangle$, equipped with the following action:

\begin{align*}
     \quad a^{-1}za &= \begin{cases}
        x & \text{if } z = x, \\
        x^{-2}y & \text{if } z = y;
    \end{cases}
    &
     \quad b^{-1}zb &= \begin{cases}
        x^{-1} & \text{if } z = x, \\
        xyx & \text{if } z = y.
    \end{cases}
\end{align*}

It is known that $Z(P_2(\mathbb{K}))=\langle b^2 \rangle$, $S(P_2(\mathbb{K}))\simeq S^1$ and $\Sigma^1(P_2(\mathbb{K}))^c=\{[\chi], [-\chi]\}$, where $\chi(x)=\chi(a)=\chi(b)=0$ and $\chi(y)=-1$. The reader may check all these facts on \cite{CarolinaWagner}, where the authors calculate the invariant.

Let $N\coloneqq \ker \chi$. Obviously $G'\leq N\leq G$ and $\rk_{\BZ} G/N= 1\leq 1=\rk_{\BZ}Z(G)$, but since $[\chi]\notin \Sigma^1(G)$ then $N$ is not finitely generated by Theorem \ref{teo.sigma}. That implies $P_2(\mathbb{K})$ is not $\FGFPa$, by Corollary \ref{corfnfpa.sigma}.

Note that the center of $P_2(\mathbb{K})$ is not a direct factor of the group (in contrast with the classical pure braid group of the disk), so the conclusion does not follow from Theorem \ref{teoLeiMaZhang}. 
\end{ex}

\section{Center as a direct factor} \label{sec.dp}

In this section we consider the case where the center of $G$ is a direct factor, \ie, $G$ is the direct product of a centerless group $H$ and a finitely generated abelian group $A$. Inspired by \cite{LeiMaZhang}, our goal here is, for each $\phi\in \Aut G$, to try to determine finiteness properties of $\Fix \phi$ based on finiteness properties of $\Fix \phi_{|{H\times 1}}$. %Since we will use Theorem \ref{Fix.Center}, we will keep the notations of Section \ref{sec.gen}. 

\begin{lema}\label{lemaauto}
    Let $H$ be a centerless group and $A$ be a finitely generated abelian group. Then every automorphism $\phi \colon H\times A \to H \times A$ has the following form:
    $$\phi(h,v)=\left( \psi(h), \alpha(h) + \gamma(v) \right),\quad (h,v)\in H\times A,$$
     where $\psi \colon H \to H$ and $\gamma \colon A\to A$ are automorphisms, and $\alpha \colon H \to A$ is a homomorphism. 
\end{lema}

\begin{proof}
    This is essentially \cite[Proposition 2.3]{LeiMaZhang}, just swapping $\BZ^n$ for $A$ finitely generated abelian. The same proof applies.
\end{proof}

 From now on in this section we write $\phi=(\psi,\alpha,\gamma)$ for the automophism $\phi$ as in Lemma \ref{lemaauto}.

     \begin{cor} \label{cordirprod}
   Let $H$ be a group of type $\F_n$ (resp. $\FP_n$) with $Z(H) = 1$ and let $\phi=(\psi,\alpha,\gamma) \colon H\times A\to H\times A$ be an automorphism, where $A$ is finitely generated abelian. Then the following assertions are equivalent:
\begin{enumerate}
    \item  $\Fix \phi$ is of type  $\F_n$ (resp. $\FP_n$),
     \item $\Fix \psi$ and $P_{\phi} = \{ h \in \Fix(\psi)  \mid \exists a \in A \text{ such that } \alpha(h) = \gamma(a) -a\}$ are of type $\F_n$ (resp. $\FP_n$),
\item $\Fix\psi$ is of type $\F_n$ (resp. $\FP_n$) and for each $\chi\in \Sigma^1(\Fix \psi)^c$ (resp. $\Sigma^1(\Fix \psi,\BZ)^c$) there is $(h,a)\in \Fix \psi\times A$ such that $\chi(h)\neq 0$ and $\alpha(h)=(\gamma - \Id)(a)$.
\end{enumerate}
\end{cor}

\begin{proof}
    Apply Theorem \ref{Fix.Center} with $G=H\times A$, noting that $Z(H\times A)=1\times A$, $\overline{\phi}=\psi$ and $\phi_{|Z(G\times A)}=\gamma$. 
\end{proof}

Now we deal with the two natural automorphisms of abelian groups: the identity and the inversion.

\begin{cor} \label{teo.gammaid}
    Let $H$ be a group of type $\F_n$ (resp. $\FP_n$) with $Z(H) = 1$ and let $A$ be a finitely generated abelian group. Let $\phi=(\psi, \alpha, \Id) \colon H \times A \to H\times A$ be an automorphism. Let $\alpha_1$ be the restriction of $\alpha$ to the subgroup $\Fix\psi$ of $H$.
 Then $\Fix \phi$ is of type $\F_n$ (resp. $\FP_n$) if and only if  $\ker \alpha_1$ is of type $\F_n$ (resp. $\FP_n$). If that is the case, then $\Fix \psi$ is of type $\F_n$ (resp. $\FP_n$).
\end{cor}

\begin{proof}
To ease notation we prove only the topological version. Note that $(h,v)\in \Fix \phi$ if and only if $h\in \Fix \psi$ and $\alpha(h)+v=v$. Hence 
\begin{equation*}\label{eqfixker}    
\Fix \phi = (\Fix \psi \cap \ker \alpha) \times A = \ker \alpha_1 \times A.
\end{equation*}

Since $A$ is $\F_{\infty}$, by Proposition \ref{basicfinprop} we have $\Fix \phi$ is $\F_n$ if and only if  $\ker \alpha_1$ is $\F_n$. If that is the case then $\Fix\psi$ is $\F_n$ by Corollary \ref{cordirprod}.
\end{proof}

\begin{ex}
Let $ G = A_{\Gamma} \times \BZ$, where $A_{\Gamma}$ is a centerless Right-angled Artin group. Then for $\alpha \colon A_\Gamma \to \BZ$ and  $\phi = (\Id, \alpha, \Id) \in \Aut(G)$, we have $\Fix \phi = \ker\alpha \times \BZ$, which by \cite{BestvinaBrady1997} may have a lot of interesting combinations of finiteness properties, e.g. it may be finitely presented but not of type $\FP_2$, or of type $\F_n$ but not $\F_{n+1}$ for any $n\geq 1$. 
\end{ex}

\begin{cor}\label{corfinind}
   Let $H$ be an  centerless group of type $\F_n$ (resp. $\FP_n$), $A$ be a finitely generated abelian group  and $\phi=(\psi,\alpha,\gamma) \colon G\times A\to G\times A$ be an automorphism such that $Fix(\gamma)$ is finite.   
   Then $\Fix \phi$ is of type $\F_n$ (resp. $\FP_n$) if and only if $\Fix \psi$ is of type  $\F_n$ (resp. $\FP_n$).
\end{cor}
\begin{proof}
     Again to ease notation we prove only the topological version. If $\Fix \phi$ is $\F_n$ then so is $\Fix \psi$ by Corollary \ref{cordirprod}.

     Now suppose $\Fix \psi$ is of type $\F_n$. Let $\alpha_1\coloneqq \alpha_{|\Fix\psi}$. Since $\Fix \gamma$ is finite then $\Fix \gamma \subset A_{tors}$, which means $(\Id_A - \gamma)(A/A_{tors})\simeq A/A_{tors}$ hence $(\Id_A - \gamma)(A)$ is a finite index subgroup of $A$. That means $P_{\phi}=\alpha_1^{-1}\left((\Id_A-\gamma)(A)\right)$ is a finite index subgroup of $\Fix \psi$ hence it is of type $\F_n$ too.

     Then $\Fix\phi$ is of type $\F_n$ by Corollary \ref{cordirprod}.
\end{proof}

\begin{cor} \label{teo.gammainv}
    Let $H$ be a centerless group of type $\F_n$ (resp. $\FP_n$), $A$ be a finitely generated  abelian group and $\phi=(\psi, \alpha, \gamma) \colon H\times A\to H\times A$ be an automorphism with $\gamma$ being the inversion. Then $\Fix \phi$ is of type $\F_n$ (resp. $\FP_n$) if and only if $\Fix \psi$ is of type $\F_n$ (resp. $\FP_n$).     
\end{cor}

\begin{proof}
  By construction, every element of $\Fix \gamma$ has order at most 2, hence $\Fix\gamma$ is finite. Then the result follows from Corollary \ref{corfinind}.
\end{proof}

The next example illustrates the case when $\gamma$ is neither the identity, nor the inversion, and $\Fix \gamma$ is infinite.

\begin{ex} \label{Ex.a}
Consider the automorphism $\gamma(x,y) = (x,-y)$ of $\BZ^2$, and let $\delta \colon H \to \BZ$ be any group homomorphism. Note that $\gamma\notin\{\Id, -\Id\}$ and $\Fix\gamma=\BZ\times 0$ is infinite. Let $\alpha \colon H \to \BZ^2$ be given by $\alpha(g) = (\delta(g),0)$. Then for $\phi = (\Id, \alpha, \gamma) \in \Aut(H \times \BZ^2)$ we have
\[ \Fix \phi = \ker(\delta) \times \BZ \times \{0\}.\]
If $\ker \delta$ is not of type $\F_n$, then nor is $\Fix\phi$, even if $\Fix\psi = \Fix\Id=H$ is of type $\F_{\infty}$.

\end{ex}

\begin{proof}[Proof of Theorem \ref{cor.dpdim}]
    We prove the topological version. Suppose $G$ has $\FnFPa$ property. Let $\psi\in\Aut H$ and let $\chi \colon H\to \BR$ be a homomorphism such that $\rk_{\BZ} \im \chi\leq \rk A$. By composing $\chi$ with an embedding $\iota  \colon \im \chi \to A$ we obtain a homomorphism $\alpha \colon H\to A$. Define $\phi\coloneqq(\psi,\alpha, \Id)\in \Aut G$, as in Section \ref{sec.dp}.  By hypothesis $\Fix \phi$ is of type $\F_n$. Applying Corollary~\ref{teo.gammaid} we obtain that $\Fix\psi$ and $\ker \alpha_{|\Fix\psi}$ are of type $\F_n$. Then $H$ has $\FnFPa$ property.  Since $\ker \alpha_{|\Fix\psi}=\ker \chi_{|\Fix\psi}$ then there is nothing else to prove.

    Now suppose the second condition.  Note that $Z(G)=1\times A$ implies $G/Z(G)\simeq H$, so let $\phi\in \Aut G$ and $\nu \colon H\to A$ be a homomorphism. By Lemma \ref{lemaauto}, there are maps $\psi\in\Aut H$, $\alpha \colon H\to A$ and $\gamma\in \Aut A$ such that $\phi=(\psi,\alpha,\gamma)$. Let $\pi \colon A\to A/I_{\phi}$ be the projection. Considering the map $\theta= \epsilon_{\phi}+\pi \circ \nu_{|\Fix\psi} \colon \Fix\psi \to A/I_{\phi}$, by Theorem \ref{FnFPaCenter} it is enough to prove that $\ker(\theta)$ is of type $\F_n$. 

 Note that $\epsilon_{\phi}=\pi\circ \alpha_{|\Fix\psi}$. Let $\beta\coloneqq \alpha+\nu \colon H\to A$ and $\beta_1 \coloneqq \beta_{|\Fix\psi}$, such that $\pi\circ \beta_1=\theta$.

Since $A/I_{\phi}$ is finitely generated abelian, there is a homomorphism $\rho \colon A/I_{\phi} \to \BR$ with finite kernel. We may consider then the composition $\chi \coloneqq \rho \circ \pi \circ \beta \colon H \to \BR$. Note that $\rk_{\BZ} \im \chi\leq \rk_{\BZ} \im \beta \leq \rk_{\BZ}A$. By hypothesis $\ker(\chi_{|\Fix \psi})$ is of type  $\F_n$.

Define the map 
\begin{align*}
\tilde{\beta} \colon \frac{\ker (\chi_{|\Fix \psi})}{\beta_1^{-1}(I_{\phi})}&\to \frac{A}{I_{\phi}}\\
\bar{h} &\mapsto \pi(\beta(h)).
\end{align*}

Note that $\tilde{\beta}$ is well defined and injective since 
$$\bar{g}=\bar{h}\Leftrightarrow \beta_1(g) -\beta_1(h)\in I_{\phi} \Leftrightarrow \tilde{\beta}(\bar{g})=\tilde{\beta}(\bar{h}).$$

Besides, the image of $\tilde{\beta}$ is inside $\ker\rho$, since $h\in \ker\chi$ implies $\chi (h)= \rho \pi \beta (h) = 0$. Hence the first quotient set is finite. 

That means $\ker(\chi|_{\Fix\psi})$ contains $\beta_1^{-1}(I_{\phi})$ as a finite index subgroup, hence $\beta_1^{-1}(I_{\phi})=\ker \pi\circ \beta_1=\ker \theta$ is also of type $\F_n$.
\end{proof}

 \section{Two counterexamples} \label{sec.ex}

 Finally we exhibit two counterexamples that establish the negative answer for Question \ref{question}, \ie, groups $H$ satisfying  $\FGFPa$ such that $H'$ is finitely generated  but $H \times \BZ$ does not satisfy $\FGFPa$. 

\subsection{First counterexample}

For the first counterexample, we need the Direct Product Formula for $\Sigma^1$.

\begin{teo}\cite{BNS1987}\label{dpformula}
    Let $G_1, G_2$ be finitely generated groups, and let $\chi \colon G_1 \times G_2 \to \BR$ be a homomorphism. Then
\[ [\chi] \in \Sigma^1(G_1 \times G_2) \iff
\begin{cases}
    [\chi_{|G_1}] \in \Sigma^1(G_1), \text{ or} \\
    [\chi_{|G_2}] \in \Sigma^1(G_2), \text{ or}\\
   \chi_{|G_1} \neq 0 \text{ and } \chi_{|G_2}\neq 0.
\end{cases}    \]
    \end{teo}

\begin{ex}  \label{conterex}
Let $N = F_2 \times F_2$. By \cite[Thm.~4.8]{ZhangVenturaWu2015}, $N$ has $\FGFPa$. Next, consider $H = N \rtimes C_2$, where the generator $\sigma$ of $C_2$ acts as $\sigma(x,y) = (y,x)$. In other words, $H$ is the wreath product  $F_2 \wr C_2$. By \cite[Thm.~9.12]{Neumann1964}, $N$ is a characteristic subgroup of $H$.

Let $\phi\in \Aut H$. Then the fixed subgroup $\Fix\phi$ contains $\Fix \phi_{|N} = \Fix\phi \cap N$ as a finite index subgroup. Since $N$ has $\FGFPa$ then $\Fix \phi$ is finitely generated. So $H$ has $\FGFPa$.

Let $[\chi]\in S(H)$. Since $\sigma$ has finite order, then $\chi(\sigma)=0$ hence  $\chi_{|N}=(\chi_1,\chi_1)$ for some character $[\chi_1]\in S(F_2)$. By Theorem \ref{dpformula}, $[\chi_{|N}] \in \Sigma^1(N)$, so Theorem \ref{teo.sigma} implies $[\chi] \in \Sigma^1(H)$. Hence $H'$ is finitely generated by Theorem \ref{teo.sigma}.

The automorphism $\psi \colon H \to H$ determined by conjugation with $\sigma$ has $\Fix \psi = C_H(\sigma) = \Delta \times C_2$, where $\Delta = \{(x,x) \in F_2 \times F_2 \mid x \in F_2\} \simeq F_2$. So $\Sigma^1(\Fix \psi) = \emptyset$ by Theorem \ref{teo.sigma}. 

Now, let $[\chi]\in S(H)$ be a character such that $\rk_{\BZ} \im \chi\leq \rk_{\BZ} \BZ=1$, \ie, a discrete character. Since $0\neq \chi_{|N}=(\chi_1,\chi_1)$ then $\chi_{|\Fix\psi}\neq 0$ hence $\chi_{|\Fix\psi}\in\Sigma^1(\Fix\psi)^c=S(\Fix\psi)$. By Theorem \ref{teo.sigma}, $\ker\chi_{|\Fix\psi}$ is not finitely generated hence $H \times \BZ$ does not have $\FGFPa$ by Theorem~\ref{cor.dpdim}. 
    
\end{ex}

\subsection{Second counterexample}
  
For the second counterexample, we will need some knowledge on Artin groups.

Given a finite simplicial graph $\Gamma$, with edges labeled by integers greater than 1, the Artin group with  $\Gamma$ as underlying graph, denoted by $A_{\Gamma}$, is given by a finite presentation, with generators corresponding to the vertices of $\Gamma$ and relations given by  \[  \underbrace{abab\cdots}_{m \text{ factors}}=\underbrace{baba\cdots}_{m \text{ factors}} \] for each edge of $\Gamma$, labeled by $m$, that links the vertices $a$ and $b$.

With that definition, we say an Artin group $A_{\Gamma}$ is of large type if all the edges of $\Gamma$ are labeled by integers greater or equal to 3. We also say that $A_{\Gamma}$ is free of infinity if $\Gamma$ is complete.

Every Artin group $A_{\Gamma}$ is associated with a Coxeter group, obtained by the quotient of $A_{\Gamma}$ modulo the normal closure of the squares of the vertices of $\Gamma$. If this Coxeter group $W$ is finite, then $A_{\Gamma}$ has a Garside element $\Delta$ such that the center of $A_{\Gamma}$ is generated by $\Delta$ or $\Delta^2$. For example, if $\Gamma$ is a single edge connecting vertices $a$ and $b$ with label $m>2$ then the Garside element of $A_{\Gamma}$ is $\Delta=\underbrace{aba\cdots}_{m \text{ factors}}=\underbrace{bab\cdots}_{m \text{ factors}}$. A good survey on Artin groups may be found at \cite{mccammond2017mysterious}.

We do not have the full description of automorphisms of Artin groups yet, but Vaskou \cite{vaskou2025automorphisms} has obtained it for large type free of infinity Artin groups, and Jones and Vaskou  \cite{JonesVaskou2024} have used this description to calculate their fixed subgroups. For our interest here, it is enough to present the result below.

\begin{cor}\label{corartinfgfpa}\cite{JonesVaskou2024}
    Let $A_{\Gamma}$ be a large type free of infinity Artin group. Then $A_{\Gamma}$ has $\FGFPa$ property. Besides, if $\psi$ is the automorphism of $A_{\Gamma}$ induced by a label-preserving graph automorphism $\sigma$, then $$\Fix\psi=A_{\Fix \sigma}\ast F,$$
   where $\Fix \sigma$ is the subgraph of fixed points of $\sigma$ and $F$ is the free group generated by the Garside elements of the groups $A_e$, for all edges $e$ whose vertices are transposed by $\sigma$. 
\end{cor} 

\begin{proof}
    Follows from \cite[Corollary 1.3]{JonesVaskou2024} and \cite[Theorem 4.4]{JonesVaskou2024}.
\end{proof}

We will also need the BNS-invariant for some Artin groups.

\begin{teo}\cite{meier1997geometric}\label{teo.meier}
  Let $A_e$ be the Artin group with a single edge $e$ as underlying graph, labeled by $m\geq 3$. Then
  \begin{enumerate}
      \item 
      If $m=2k$, $k>1$, then $S(A_e)=S^1$ and $\Sigma ^1(G_e)=S^1\setminus \{(1,-1), (-1,1)\}$.
\item
If $m=2k+1$ then $\Sigma^1(A_e)=S(A_e)=\{\pm1\}$.
  \end{enumerate}
\end{teo}

In the hypothesis of Theorem \ref{teo.meier}, if the endpoints of $e$ are $v$ and $w$,  then $[\chi]\in \Sigma ^1(A_e)^c$ if and only if $m$ is even and $\chi(v)=-\chi(w)\neq 0$. We will name the edges described above as ``$\chi$-dead edges''. We also say a subgraph $\lgoth$ of $\Gamma$ is dominant if every vertex of $\Gamma$ is adjacent to some vertex of $\lgoth$.

\begin{teo}\cite{almeida2015sigma1}\label{teo.almeida}
    Let $A_{\Gamma}$ be an Artin group such that $\Gamma$ has circuit rank 1 (\ie, $\pi_1(\Gamma)$ is infinite cyclic). Define the living subgraph $\lgoth=\lgoth(\chi)$ as being the subgraph obtained from $\Gamma$ after removing all vertices $v\in V(\Gamma)$ such that $\chi(v)=0$ and all the $\chi$-dead edges. Then
    $$\Sigma^1(A_{\Gamma})=\{[\chi]\in S(A_{\Gamma}) \tq \lgoth(\chi) \text{ is a connected and dominant subgraph of $\Gamma$} \}.$$ 
\end{teo}

Theorem \ref{teo.almeida} is actually part of an ongoing general conjecture for Artin groups, with some recent advancements (\cf \cite{ferrer2025sigma}). Now we can proceed to the second counterexample.

\begin{ex}\label{exartin}
Let $\Gamma$ be the graph
$$
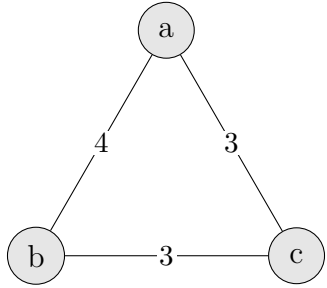
\begin{tikzpicture}[
    vtx/.style={circle, draw, fill=gray!20, minimum size=1.8em},
    lbl/.style={fill=white, inner sep=1pt, font=\small}
]

    \node[vtx] (a) at (0, 2) {a};
    \node[vtx] (b) at (-1.732, -1) {b};
    \node[vtx] (c) at (1.732, -1) {c};

    \draw (a) -- (b) node[midway, lbl] {4};
    \draw (b) -- (c) node[midway, lbl] {3};
    \draw (c) -- (a) node[midway, lbl] {3};
\end{tikzpicture}$$
and let $H \coloneqq A_{\Gamma}=\apres{a,b,c}{aca=cac,bcb=cbc,abab=baba}$, a free of infinity large type Artin group. Then $H$ has FGFPa  by Corollary \ref{corartinfgfpa} and it is centerless since it is large-type of rank 3 (\cf \cite[Remark 2.11]{JonesVaskou2024}). 

To calculate the BNS-invariant of $H$, note that, because of the Artin group presentation, for each $[\chi]\in S(H)$ we have $\chi(a)=\chi(b)=\chi(c)\neq 0$, so $\Sigma^1(H)=S(H)=\{\pm 1\}$ by Theorem \ref{teo.almeida}, hence $H'$ is finitely generated by Theorem \ref{teo.sigma}.

On the other hand, consider the automorphism $\psi\in \Aut H$ induced by the graph automorphism $\sigma \colon \Gamma\to\Gamma$ given by $\sigma(a)=b$, $\sigma(b)=a$ and $\sigma(c)=c$. By Corollary \ref{corartinfgfpa}, $\Fix \psi = \langle c \rangle \ast \langle abab \rangle$, which is free hence $\Sigma^1(\Fix\psi)=\emptyset$ by Theorem \ref{teo.sigma}.

Now consider $\chi \colon H\to \BR$ given by $\chi(a)=\chi(b)=\chi(c)=1$. Then $\chi_{|\Fix\sigma}\neq 0$ hence $[\chi_{|\Fix\psi}]\in S(\Fix\psi)=\Sigma^1(\Fix\psi)^c$. By Theorem \ref{cor.dpdim},  $H\times \BZ$ does not have FGFPa.

\end{ex}

\section*{Acknowledgements}
This work was partially supported by FINAPESQ-UEFS and grant FAPEMIG [APQ-02750-24].

\end{document}